\documentclass[preprint,12pt]{elsarticle} 

\usepackage{amsmath,amssymb,amsthm} 
\usepackage{mathtools}              
\usepackage{bm}                     
\usepackage{graphicx}               
\usepackage{hyperref}               
\usepackage{enumitem}               
\usepackage{caption,subcaption}     

\theoremstyle{plain}
\newtheorem{theorem}{Theorem}[section]

\newtheorem{proposition}[theorem]{Proposition}
\newtheorem{corollary}[theorem]{Corollary}

\theoremstyle{definition}
\newtheorem{definition}[theorem]{Definition}
\newtheorem{example}[theorem]{Example}

\theoremstyle{remark}
\newtheorem{remark}[theorem]{Remark}

\begin{document}

\begin{frontmatter}

\title{Restricted Liouville Operator for the study of Non-Analytic Dynamics within the Disk}

\author[inst1]{Sushant Pokhriyal\corref{cor1}}
\ead{spokhriyal@usf.edu}
\author[inst1]{Joel A Rosenfeld}
\ead{rosenfeldj@usf.edu}
\address[inst1]{Department of Mathematics and Statistics, University of South Florida, 4202 E Fowler Avenue, Tampa, USA}
\cortext[cor1]{Corresponding author. Email: spokhriyal@usf.edu}

\begin{abstract}
The study of Koopman and Liouville operators over reproducing kernel Hilbert spaces (RKHSs) has been gaining considerable interest over the past decade. In particular, these operators represent nonlinear dynamical systems, and through the study of these operators, methods of system identification and approximation can be derived through the exploitation of the linearity of these systems. The resulting algorithms, such as Dynamic Mode Decompositions, can then make predictions about the finite-dimensional nonlinear dynamics through a linear model in infinite dimensions. However, considering bounded and densely defined Koopman and Liouville operators over RKHSs often restricts the dynamics to those whose smoothness or analyticity matches that of the functions within that space. To circumvent this limitation, this manuscript introduces the \textit{Restricted Liouville Operators} over the Hardy space on unit disc, which will allow for a wider class of dynamics (non-analytic or non-smooth) than available. 
\end{abstract}


\end{frontmatter}
\section{Introduction}\label{sec1}

The study of nonlinear dynamical systems through linear operators can be traced back to Koopman and von Neumann's work in the 1930s \cite{koopman1931hamiltonian, koopman1932dynamical}. More recently, the introduction of Dynamic Mode Decompositions (DMD) by Schmidt and Sesterhenn in 2008, \cite{schmid2010dynamic}, opened the door to data-driven analysis of nonlinear dynamical systems through linear operators. Their work was later expanded upon through extended DMD (EDMD) by Williams et al in \cite{williams2015data}, as a method that combines the selection of nonlinear features with DMD was extended to infinite feature spaces, and since then, reproducing kernel Hilbert spaces (RKHSs) have been an integral part of the study of DMD . Mezic and Rowley later connected the study of DMD with ergodic theory and the Koopman operator in \cite{budivsic2012applied, mezic2020spectrum, williams2014kernel}. For detailed review of these techniques, please refer to \cite{brunton2022modern, colbrook2024multiverse}, and references therein.

In Gonzalez et al. \cite{gonzalez2024kernel}, it was observed that by combining boundedness assumptions on Koopman operators together with regularity conditions of members of an RKHS, very strong conditions on the dynamics are imposed. For instance, all bounded Koopman operators over the Bargmann-Fock space have been shown to have affine symbols (dynamics) in Carswell et al. \cite{carswell2003composition}. This same degeneracy also occurs when working over the native space of a Gaussian radial basis function, the RKHS of polynomials, etc. This motivates the use of other operators to represent non-linear dynamics over RKHSs, such as the Koopman generator (or more generally, Liouville operators). While Koopman generators presuppose forward completeness of nonlinear dynamics so that discrete-time dynamics may be generated, Liouville operators do not have this requirement but have the same formal definition as Koopman generators \cite{rosenfeld2023singular}.

While often unbounded, Liouville operators are densely defined for large classes of dynamics \cite{rosenfeld2024occupation}. In Rosenfeld et al. \cite{rosenfeld2023singular}, scaled Liouville operators were introduced, which are compact for a large collection of dynamics, and in Rosenfeld et al., the singular DMD method was introduced to provide a collection of compact Liouville operators without a scaling factor. The compactness guarantees are important in these settings, since the DMD method approximates dynamic operators through finite rank approximations derived from the available trajectory data of the non-linear dynamical systems.

Recently, work by Mezic and Mauroy have been focused on the development of Koopman operators corresponding to analytic dynamics over the disc through the Hardy space \cite{mauroy2024analytic}, and work by Russo and Rosenfeld has begun to study Liouville operators over the Hardy space \cite{russo2022liouville}. Despite the expansion of the dynamics considered admissible for DMD methods through Liouville operators, the dynamics are required to match the regularity properties of the functions of the spaces that they operate over [cite]. The removal of this limitation will allow for the consideration of a wider class of dynamics, while also retaining many of the benefits of RKHSs, such as bounded point evaluations. In the case of dynamical system which involves some non-analyticity, we will carefully construct a Hilbert space over which the Liouville operator is well defined. Specifically, we will define a symbol in $H^\infty$, that nullify the effect of non-analytic part of the dynamical system. 

 In this vein, the present manuscript aims to study restrictions of Liouville operators, to broaden the class of admissible dynamics to study through dynamic operators. In Section \ref{Sec 2}, we review the the motivation behind working with this particular manuscript. In Section \ref{Sec 3}, we will review the background needed for this paper and introduce the main problem we want to tackle. In Section \ref{Sec 4}, Restricted Liouville operators is introduced. In Section \ref{Sec 5}, its connections to occupation kernels and data driven methods for dynamical systems are explored. In Section \ref{Sec 6}, we will talk about the spectrum of the restricted Liouville operator, which gives insight into explaining the dynamical system in question. Section \ref{Sec 7} presents an approximation method for estimating the symbol of the Restricted Liouville operator.

\section{Motivation}\label{Sec 2}

Consider the dynamical system

$$
\dot{x}=f(x)
$$
defined over the state space $X$, and $f:X\rightarrow X$ be any function. Clearly, the behavior of the dynamical system depends on $f$. For example, if we have a n-dimensional linear system of differential equations (
$\dot{x}=Ax$) with a single fixed point at the origin we can observe several types of behaviors, such as saddle points, spirals, cycles, stars and nodes, which are well-understood. We classify these cases based on the eigenvalues of the matrix $A$ used to classify the system. With a nonlinear system the behavior of the system is more difficult to analyze. 

Fortunately, we are not left completely in the dark. There are ways to make a non-linear dynamical system into linear systems, at the cost of moving from the finite dimensional setting to infinite dimensional one. In general, this can be achieved by looking at the dynamics under the action of some suitable linear operator (possibly unbounded), like Koopman, Perron-Frobeinous, Liouville, Carleman Linearization, etc, on some Hilbert space of functions over the state space $X$. In all of these methods, we seek to approximate the unknown dynamics $f$ through the projection of $f$ over a set of suitable basis functions, which are mostly analytic in the Hilbert space on $X$. This forces regularity condition on $f$. 

Consider the following example of two-compartment model to describe the drug distribution (Example 7 in \cite{strasser2021data});
\begin{align*}
    \dot{x_1} &= -\cfrac{x_1}{5+x_1}+x_1 +x_2+u\\
    \dot{x_2} &= x_1 - x_2
\end{align*}
where $(x_1(t),x_2(t))\in \mathbb{R}^2$ is the state of the system and $u(t)$ is the control input for all $t\geq 0$. Notice, the equation ODE involves rational dynamics. 

There are important class of non-linear systems, where $f$ is beyond polynomials or analytic functions, for example, see Example 6 and 7 of drug distribution in \cite{strasser2021data}, enzyme kinematics \cite{holmberg1982practical} or biochemical reactors \cite{strogatz2024nonlinear}, and references therein. In general, identification of such systems can be challenging \cite{evans2002identifiability}. In addition to this,  problems in control can also give rise to system of differential equations in which the function $f$ can have discontinuities like finite jumps (switching on or off at certain time) or sudden changes in the behavior of the system \cite{chappell1990global}. Once during a test or simulation if a point of discontinuity is located, then our goal is to design the learning algorithm in such a way that it takes care of such discontinuous points by associating the symbol $f$ with suitable Hilbert space (See Figure \ref{fig:Traj_pole}). It is our desire to explore this avenue and start our quest of deep analysis towards this direction. Although, it is always not possible to figure out all such points during simulations, nonetheless we will start our investigation assuming that such points can be traced. The following items are the main contribution of this paper;
\begin{itemize}
    \item The case of non-analytic dynamics is considered. In particular, we will talk about the case when $f=p/q$, where $p$ and $q$ are polynomials, and zeroes of $q$ are inside the disc.
    \item A suitable Hilbert space is constructed over which the Liouville operator can be leveraged to learn the unknown dynamics $f$.
    \item A new inner product is introduced using the rational symbols, to optimize the parametric identification of the unknown system.
\end{itemize}

In this paper, we will explore the dynamics inside the state space of unit disc $\mathbb{D}$, as we can always normalize the collected data inside such domain. Another reason for this assumption is the availability of rich theory of Reproducing kernel Hilbert spaces (RKHS) associated with $\mathbb{D}$.

\begin{figure}[h]
\centering
\includegraphics[width=.6\textwidth]{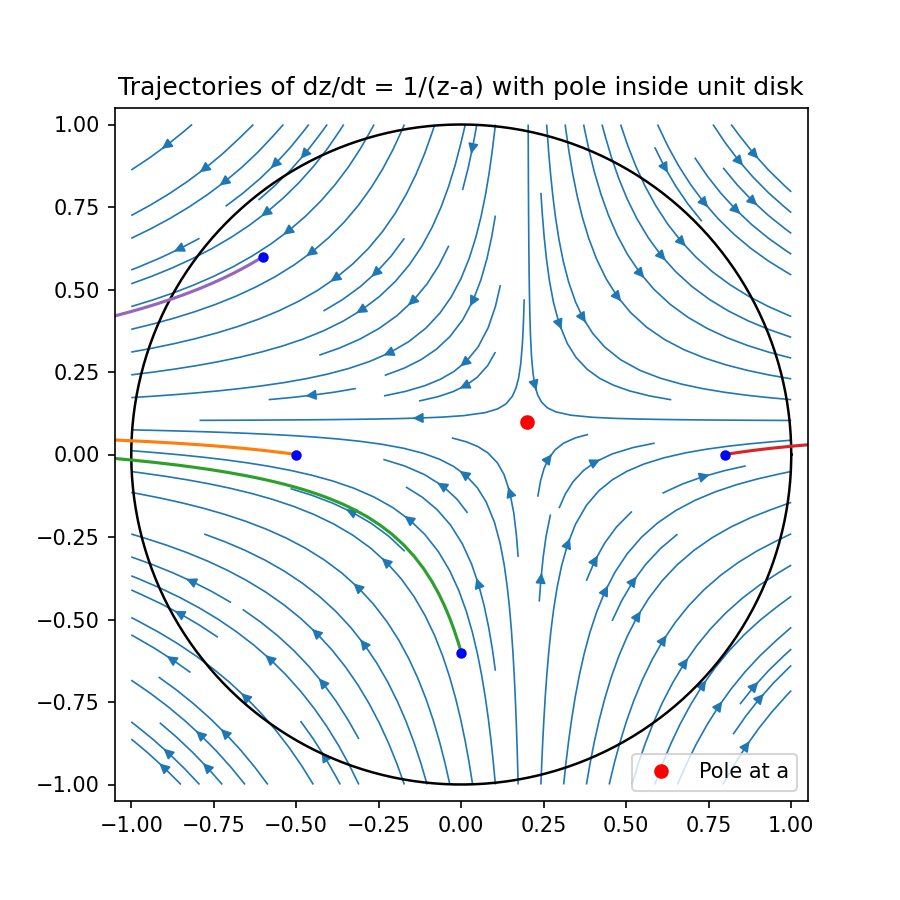}
\caption{The trajectory of the dynamics $\dot{z}= \frac{1}{z-a}$ for $a=(0.2,0.1)$}
\label{fig:Traj_pole}
\end{figure}

\section{Background and Problem Formulation}\label{Sec 3}
\subsection{Reproducing Kernel Hilbert Space: RKHS}
\begin{definition}
    A Reproducing Kernel Hilbert Space (RKHS), $H$, over a set $X$ is a Hilbert space of functions that map $X$ to $\mathbb{C}$ such that for each $x \in X$, the evaluation functional $f \mapsto f(x)$ is bounded.
\end{definition}

By the Riesz representation theorem, for each $x \in X$ there exists a function $k_x \in H$ such that $f(x) = \langle g, k_x \rangle_H$ for all $g \in H$ called the kernel function centered at $x$. The kernel function, given as $K(x,y) = \langle k_y, k_x \rangle_H$ is called the kernel function associated with the RKHS $H$.

One of the well studied RKHS is knwon as the Hardy-Hilbert space of analytic functions over unit disk $\mathbb{D}$, defined as follows

$$
H^2(\mathbb{D})=\bigg\{h:\mathbb{D}\rightarrow \mathbb{C}: h(z)=\sum_{n=0}^{\infty} h_n z^n \text{~and~}\sum_{n=0}^{\infty}|h_n|^2<\infty \bigg\}
$$

It is well known that for all $h\in H^2(\mathbb{D})$, the non-tangential or radial limit exist, that is  $\text{lim}_{r\rightarrow 1^{-}} h_r(e^{i\theta})=\text{lim}_{r\rightarrow 1^{-}}h(re^{i\theta})= \tilde{h}(e^{i\theta})$ a.e for all $e^{i\theta}\in \mathbb{T}$. This provide us an another approach to identify the functions $h \in H^2(\mathbb{D})$, as the function $\tilde{h}\in L^2(\mathbb{T})$, whose fourier coefficients $\hat{h}(n)=\int_{0}^{2\pi} \tilde{h}(e^{i\theta})e^{-in\theta}d\theta$ for all $n\geq 0$. More precisely, we can define

$$
H^2(\mathbb{T})= \bigg \{h\in L^2(\mathbb{T}): h(e^{i\theta})=\sum_{n=0}^{\infty} h_n e^{in\theta} \bigg\}
$$
It is the closed subspace of $L^2(\mathbb{T})$. Vice versa, we can define the functions in the Hardy-Hilbert space $H^2(\mathbb{D})$, as the integral of the functions in  $H^2(\mathbb{T})$ against the Poisson Kernel \cite{fricain2016theory}. The reproducing kernel of $H^2(\mathbb{D})$ is known as the Szego Kernel and it is given by 
$$
K(z,w)=k_w(z)=\cfrac{1}{1-\bar{w}z}
$$

Kernel functions interact nicely with many classical function theoretic operators. For example, if $f : X \to \mathbb{C}$ is the symbol for a densely defined multiplication operator, $M_f : \mathcal{D}(M_f) \to H$, with $\mathcal{D}(M_f) := \{ g \in H : fg \in H \}$ and $M_f g := gf$, then $M_f^* k_x = \overline{f(x)} k_x$. That is kernel functions are eigenfunctions for the adjoints of densely defined multiplication operators \cite{paulsen2016introduction}. Similarly, the adjoint of a densely defined Koopman or composition operator, $\mathcal{K}_\phi$, on a kernel function is described as $\mathcal{K}_\phi^* k_x = k_{\phi(x)}$.\\

\subsection{Liouville Operator and Occupation Kernel}
\begin{definition}
  Given a function $f:\mathbb{D}\rightarrow \mathbb{C}$, the \textit{Liouville Operator} with symbol $f$, is defined as $$A_f:D(A_f)\rightarrow H^2 \text{~such that~}A_f g(\cdot)=f(\cdot)\cfrac{d}{dz}g(\cdot)$$ 
\end{definition}
 where $D(A_f)=\Big\{g\in H^2:f(\cdot)\cfrac{d}{dz}g(\cdot)\in H^2\Big\}$.\\

 It is easy to check that $A_f$ is a linear operator. Being a differential operator, $A_f$ is expected to be unbounded operator over $H^2(\mathbb{D})$, but as differential operators over RKHS's consisting of continuously differentiable functions, it is always closed \cite{rosenfeld2024occupation}, hence $A_f$ is a closed operator over $H^2(\mathbb{D)}$. 

\begin{definition}
    Let $T>0$ and suppose $\theta :[0.T]\rightarrow \mathbb{D}$ defines a continuous signal in $\mathbb{D}$. The functional on $H^2(\mathbb{D})$ given by $g\rightarrow \int_{0}^{T}g(\theta(t))dt$ is bounded, and hence, there is a function in $H^2$, deonted by $\Gamma_{\theta}$, such that $\int_{0}^{T}g(\theta(t))dt=\langle g, \Gamma_{\theta}\rangle_{H^2}$. The function $\Gamma_{\theta}$ is called the \textit{Occupation Kernel} corresponding to $\theta$ in $H^2$. 
\end{definition}

When $\gamma:[0,T]\rightarrow \mathbb{D}$ is a trajectory satisfying the dynamics $\dot{z}=f(z)$, then it can be shown that
$$
\Gamma_{\gamma}\in D(A_{f}^{*}),\text{~and~} A_{f}^{*}\Gamma_{\gamma}= K_{\gamma(T)}-K_{\gamma(0)}.
$$

The above relation can be used to find finite approximation of the Liouville operator. For more information, we refer to the papers \cite{rosenfeld2024occupation}.

\subsection{Main Problem}
Suppose for a non-linear dynamical system $\dot{z}=f(z)$, we collect trajectories $\{\gamma_j: [0,T]\rightarrow \mathbb{D}\}_{j=1}^{N}$ such that $\dot{\gamma_j}=f(\gamma_j)$ for $j=1,\ldots,N$. The goal of this paper is to parametrize the function $f$ in terms of basis functions in $Y_i:\mathbb{D}\rightarrow \mathbb{D}$ for $i=1,\ldots, M$ such that
\begin{equation}\label{eq 1}
\dot{z}= f(z)=\sum_{i=1}^{M}\theta_i Y_i(z)
\end{equation}

Since $\gamma_j$ is a solution of the equation \ref{eq 1}, then $\gamma_j(T)-\gamma_j(0)=\int_{0}^{T}\sum_{i=1}^{M}\theta_i Y_i(\gamma_j(t))dt$. A solution of the system \ref{eq 1} is to solve the following minimization problem;

$$
\boxed{\underset{\theta_1,\ldots,\theta_M}{\text{min}}\sum_{i=1}^{N}\Bigg\|\gamma_j(T)-\gamma_j(0)-\sum_{i=1}^{M}\theta_i \int_{0}^{T}Y_i(\gamma_j(t))dt\Bigg\|_{2}^{2}}
$$

\section{Restricted Liouville Operator}\label{Sec 4}

In this paper, we will pose the above problem inside $H^2(\mathbb{D})$, by taking $f$ as a symbol which is not analytic or smooth in nature. The following theorem from \cite{russo2022liouville}, explains the behavior of the symbol $f$ of the operator $A_f$, when $\overline{D(A_f)}=H^2(\mathbb{D})$.
 \begin{theorem}\label{thm rosen-russo}
     If $f$ is a symbol for a densely defined operator $A_f$,then $f$ is analytic, and $f=\cfrac{b}{a\frac{d\phi}{dz}}$, where $b\in H^2$, $a$ is a outer function, and $\phi$ a function in BMOA.
 \end{theorem}
The requirement that the operator $A_f$	be densely defined imposes an analyticity condition on the symbol of the dynamics, significantly restricting the scope of our investigation. In particular, this constraint excludes many dynamical systems governed by non-analytic or non-smooth symbols, like rational functions with poles inside the state space.

Clearly, when $f(z)=p(z)/q(z)$, where $p(z), q(z)$ are polynomials, then Theorem \ref{thm rosen-russo} implies that the domain of the Liouville Operator $A_f$ cannot be densely defined. Hence, we need to impose some other conditions to make sure that the domain of $A_f$ can de densely defined. This motivates the definition of \textit{Restricted Liouville Operator}. Let us go through the following example to show what we actually mean.
\begin{example}
    Consider $f(z)=1/z, ~z\in \mathbb{D}$, then it can be easily seen that the monomials $\{z^n\}_{n=0,n\neq 1}^{\infty} \subset D(A_{1/z})$. However, $z\notin D(A_{1/z})$, otherwise $\overline{D(A_{1/z})}=H^2(\mathbb{D})$, which is a contradiction to the Theorem \ref{thm rosen-russo}. On the other hand, consider $D(A_{1/z}):z^2H^2\rightarrow H^2$, then $\overline{D(A_{1/z})|_{z^2H^2}}=z^2H^2$. Hence the restriction of the Liouville operator on the appropriate domain can be densely defined.
\end{example}
The above example give us hope to define the adjoint of Liouville operator on appropriate Hilbert space, as the domain for this operator can be dense even in the case of non-analytic symbol $f$. 

Now, we will construct the appropriate Hilbert space such that the Liouville operator is densely defined. To achieve this, we will define a very well known function, which kind of nullify the poles of the function $f$. Let $a_1,a_2,\ldots,a_n,$ be the zeroes of the function $q(z)$ inside $\mathbb{D}$ away from the boundary of the disk $\partial(\mathbb{D})=\mathbb{T}$, or in other words these points are the poles of the function $f$. Consider the following functions
$$
B_j(z)= \cfrac{z-a_j}{1-\overline{a_j} z}, \text{~for all~} j=1,2,\ldots,n.
$$
Then it can be easily proved that $B_j\in H^{\infty}(\mathbb{D})$ and $|B_j(e^{i\theta})|=1$ for all $e^{i\theta}\in \mathbb{T}$. In literature, the function $B_j$'s are known as the finite Blaschke factors. One can easily show that $\|B_j g\|_{H^2}=\|g\|_{H^2}$ for all $g\in H^2$, and hence multiplication operator $T_{B_{j}}:H^2\rightarrow H^2$ is an isometry, which implies that the range space $T_{B_{i}}H^2=B_i H^2$ is an closed subspace of $H^2$ (see, \cite{lata2022multivariable}). 

Let us define $B(z)=B_1(z)B_2(z)\cdots B_n(z)=\prod_{k=1}^{n}\cfrac{z-a_k}{1-\overline{a_k}z}=\prod_{k=1}^{n}\cfrac{q(z)}{1-\overline{a_k}z}$, and $R_j (z)=\prod_{k=1, k\neq j}^{n}\cfrac{z-a_k}{1-\overline{a_k}z}$, then direct computation shows that
$$
B'(z)=\prod_{k=1}^{n}\cfrac{1-|a_k|^2}{(1-\overline{a_k}z)^2}R_k(z)
$$ 
Moreover, it is easy to see that $B\in H^{\infty}(\mathbb{D})$ and $|B(e^{i\theta})|=1$ for all $e^{i\theta}\in \mathbb{T}$ (see, \cite{fricain2016theory}). Now, we are in the position to construct an Hilbert space which captures the singularities of the function $f$. Define the set 
$$B^2H^2=\{B^2g:g\in H^2,\text{~with~} \|B^2 g\|_{H^2}=\|g\|_{H^2}\}$$

 Then $B^2 H^2$ is a closed subspace of $H^2$ with $\{B^2 z^n\}_{n=0}^{\infty}$ as its orthonormal basis \cite{lata2022multivariable}. Moreover, according to \cite{garcia2018finite}, the space $B^2H^2$ is an RKHS with the reproducing kernel defined as
 \[
 K_{B^2}(w,z)=\cfrac{\overline{B^2(w)}B^2(z)}{1-\overline{w}z}=\overline{B^2(w)}B^2(z)K_w(z).
 \]
 Now, the following result is immediate.

\begin{proposition}
    The Restricted Lioville operator representated by $A_{f_{|B^2 H^2}}: B^2H^2\rightarrow H^2$ is densely defined.
\end{proposition}
\begin{proof}
    For $n\in \mathbb{N}$, we will prove that $B^2 z^n\in D(A_{f_{|B^2 H^2}})$. Consider
    \begin{align*}
   f(z) \cfrac{d}{dz}(B^2(z)\cdot z^n) &= f(z)\big[(nz^{n-1}B^2(z) + 2B(z)B'(z)\cdot z^n\big]\\  
    & = f(z)B(z)\big[nB(z)\cdot z^{n-1}+B'(z)\cdot z^n\big]\\
    &= \cfrac{p(z)}{q(z)}B(z)\big[nB(z)\cdot z^{n-1}+B'(z)\cdot z^n\big]\\
    &= \prod_{k=1}^{n}\cfrac{p(z)}{1-\overline{a_k}z}\big[nB(z)\cdot z^{n-1}+B'(z)\cdot z^n\big]  
    \end{align*}
For each $k$, since $|a_k|<1$ and $|z|<1$, we have,
$
|1-\overline{a_k}z|\geq 1-|a_k||z|> 1-|a_k|
$
,then $$\text{sup}_{z\in \mathbb{D}}\bigg| \prod_{k=1}^{n}\cfrac{p(z)}{1-\overline{a_k}z}\bigg|\leq \prod_{k=1}^{n}\cfrac{\|p\|_{\infty}}{1-|a_k|}<\infty.$$ 
Also, it can be proved that $$|B'(z)|^2\leq\sum_{k=1}^{n}\cfrac{1+|a_k|}{1-|a_k|}<\infty$$
This means that for all fixed $n\in \mathbb{N}$, we have $f(z) \frac{d}{dz}(B^2(z)\cdot z^n)\in H^{\infty}$. In particular, this implies that $B^2(z)\cdot z^n\in D(A_{f_{|B^2 H^2}})$, and hence $\overline{D(A_{f_{|B^2 H^2}})}=B^2H^2$. This completes the proof. 
\end{proof}

\begin{remark}
   In the case of poles having multiplicity greater than or equal to 1, we can similarly define the function $B(z)=B_1^{m_1}(z)B_2^{m_2}(z)\cdots B_n^{m_n}(z)$, and come with the same conclusion as the above theorem. 
\end{remark}
 Clearly, the above construct does not work for the functions where the singularity of $f$ are very close to the boundary of $\mathbb{D}$ or at the boundary itself, for example, $f(z)=1/(z-1)$. Thus, the analysis of such non-analytic functions will be reserved for the future work.

\section{Action of the adjoint of Restricted Liouville Operator}\label{Sec 5}
As we have already established that for non-analytic symbol $f(z)=p(z)/q(z)$, the Restricted Lioville operator $D(A_{f_{|B^2 H^2}}): B^2H^2\rightarrow H^2$ is densely defined, then it possesses a well defined adjoint. Then, we can prove the following.

\begin{proposition}\label{pro 4.1}
    The occupation kernel corresponding to $\gamma:[0,T]\rightarrow \mathbb{D}$ satisfying the dynamics $\dot{z}=f(z)$, is contained in the domain of the adjoint of restricted Liouville operator, that is $\Gamma_{\gamma}\in D(A_{f_{|B^2 H^2}}^{*})$.
\end{proposition}
\begin{proof}
     To see this, let $B^2 g\in D(A_{f_{|B^2 H^2}})$, then

\begin{align*}
    \langle A_{f_{|B^2 H^2}} (B^2 g),\Gamma_{\gamma}\rangle_{H^2} &= \int_{0}^{T}f\big(\gamma(t)\big)\cfrac{d}{dz}\Big(B^2\big(\gamma(t)\big)g\big(\gamma(t)\big)\Big)dt\\
    &= \int_{0}^{T}\dot{\gamma(t)}\cfrac{d}{dz}\Big(B^2g\big(\gamma(t)\big)\Big)dt\\
    &= \int_{0}^{T}\cfrac{d}{dt}\Big(B^2g\big(\gamma(t)\big)\Big)dt\\
    & =  B^2g(\gamma(T))- B^2g(\gamma(0))\\
    & = \langle B^2 g, K_{B^2}(\gamma(T),\cdot)\rangle_{H^2}-\langle B^2 g, K_{B^2}(\gamma(0),\cdot)\rangle_{H^2}\\
    \end{align*}
Hence, the functional $B^2 g\hookrightarrow \langle A_{f_{|B^2 H^2}} (B^2 g),\Gamma_{\gamma}\rangle_{H^2}$ is bounded by Cauchy-Schwarz inequality and the result follows.
\end{proof}

The action of $A_{f_{|B^2 H^2}}^{*}$ over the $K_{w}^{[j-1]}$ are essential for our understanding, as the following result says that $K_{w}^{[j-1]}$ are the eigenvectors for our operator.
\begin{proposition}\label{pro 4.2}
    Let $w\in \mathbb{D}$ and $K_w(z)=\cfrac{1}{1-\overline{w}z}$ be the Szeg$\ddot{o}$ kernel of $H^2$, then for all $j\in \mathbb{N}$, the derivative $K_{w}^{[j-1]}(z)=\cfrac{d^j}{d\overline{w}^j}\bigg(\cfrac{1}{1-\overline{w}z}\bigg)\in D(A_{f_{|B^2 H^2}}^{*})$, and $$A_{f_{|B^2 H^2}}^{*}K_{w}^{[j-1]}=\sum_{l=0}^{j-1}\binom{j-1}{l}\overline{f^l (w)}K_{w}^{[j-1]}  $$
\end{proposition}
\begin{proof}
  Let $w\in \mathbb{D}$ and set
  $$
K_{w}^{[j]}(z)=\cfrac{d^j}{d\overline{w}^j}\bigg(\cfrac{1}{1-\overline{w}z}\bigg)=\sum_{n=j}^{\infty}\cfrac{n!}{(n-j)!}z^n~ \overline{w}^{n-j}
  $$
  then $K_{w}^{[j]}\in H^2$. Since for any $h\in H^2\Rightarrow h(w)=\langle h, K_w\rangle_{H^2}$, then $h^{[j]}(w)=\langle h, K_{w}^{[j]}\rangle_{H^2}$. Suppose $B^2g\in D(A_{f_{|B^2 H^2}})$, and consider
  \begin{align*}
      \langle A_{f_{|B^2 H^2}} (B^2 g),K_{w}^{[j-1]}\rangle_{H^2} &= \langle f (B^2 g)',K_{w}^{[j-1]}\rangle_{H^2} =\big((B^2g)'f\big)^{[j-1]}(w)\\
      &= \sum_{l=0}^{j-1}\binom{j-1}{l}(B^2g)^{[j-l]} (w)f^{l}(w)\\
      &= \langle B^2g, \sum_{l=0}^{j-1}\binom{j-1}{l}\overline{f^l (w)}K_{w}^{[j-1]}\rangle_{H^2}
  \end{align*}
SInce the above equality is true for all $B^2g\in B^2H^2$,  thus, $K_{w}^{[j-1]}\in D(A_{f_{|B^2 H^2}}^{*})$, and $$A_{f_{|B^2 H^2}}^{*}K_{w}^{[j-1]}=\sum_{l=0}^{j-1}\binom{j-1}{l}\overline{f^l (w)}K_{w}^{[j-1]}.  $$

\end{proof}

 The following result establish the action of the adjoint of Restricted Liouville operator on vectors related to kernel functions. 
\begin{theorem} \label{Thm 4.3}
    Suppose $\gamma:[0,T]\rightarrow \mathbb{D}$ satisfying the dynamics $\dot{z}=f(z)$, then $A_{f_{|B^2 H^2}}^{*}\Gamma_{\gamma}(\cdot)=K_{B^2}(\gamma(T),\cdot)-K_{B^2}(\gamma(0),\cdot)$. More generally, for a continuous trajectory $\theta:[0,T]\rightarrow \mathbb{D}$, we have $ A_{f_{|B^2 H^2}}^{*}\Gamma_{\theta}= \int_{0}^{T}\overline{f(\theta(t))K_{\theta(t)}^{[1]}}~dt$.
\end{theorem}

\begin{proof}
    For any $B^2 g\in D(A_{f_{|B^2 H^2}})$, using the Proposition \ref{pro 4.1}, we can conclude that;
    \begin{align*}
        \langle A_{f_{|B^2 H^2}} (B^2 g),\Gamma_{\gamma}\rangle_{H^2} &= \langle B^2 g, K_{B^2}(\gamma(T),\cdot)- K_{B^2}(\gamma(0),\cdot)\rangle_{H^2}\\
        \langle  B^2 g, A_{f_{|B^2 H^2}}^{*}\Gamma_{\gamma}\rangle_{H^2}&= \langle B^2 g, K_{B^2}(\gamma(T),\cdot)- K_{B^2}(\gamma(0),\cdot)\rangle_{H^2}\\
       \Rightarrow~~~~~ A_{f_{|B^2 H^2}}^{*}\Gamma_{\gamma} &= K_{B^2}(\gamma(T),\cdot)- K_{B^2}(\gamma(0),\cdot)
    \end{align*}
 For the action of adjoint of $A_{f_{|B^2 H^2}}^{*}$ on $\Gamma_{\theta}$, note that 
 $$B^2 g\hookrightarrow \langle A_{f_{|B^2 H^2}} (B^2 g),\Gamma_{\theta}\rangle_{H^2}=\int_{0}^{T}f(\theta(t))(B^2(\theta(t) g(\theta(t)))'dt$$ is bounded and $B^2g^{[1]}(w)=\langle B^2g, K_{w}^{[1]}\rangle_{H^2}$. Hence,
 $$\langle A_{f_{|B^2 H^2}} (B^2 g),\Gamma_{\theta}\rangle_{H^2}=\int_{0}^{T}f(\theta(t))(B^2(\theta(t) g(\theta(t)))'dt= \int_{0}^{T}f(\theta(t))\langle B^2g, K_{\theta(t)}^{[1]}\rangle_{H^2}~dt.$$ Therefore,
 $$ A_{f_{|B^2 H^2}}^{*}\Gamma_{\theta}= \int_{0}^{T}\overline{f(\theta(t))K_{\theta(t)}^{[1]}}~dt$$
\end{proof}
The above result is crucial, as it intertwines the Liouville operator with the occupation kernel and the corresponding dynamics.

\section{Spectrum of Restricted Liouville Operator}\label{Sec 6}

As the eigenvalues and eigenfunctions of the differential operator reveals the nature of dynamical system, so it is very important to learn the spectrum of Restricted Liouville operator. Since the spectrum of unbounded operator can be empty or whole complex plane, so we need to learn the condition on the symbol $f$ which can show such erratic behavior. Let's start with the following case, where the symbol is $f(z)=1/z$. 

\begin{proposition}\label{pro 5.1}
    The point spectrum of $A_{1/z}$ is empty.
\end{proposition}
\begin{proof}
    Suppose $\lambda$ is the eigenvalue of the operator $A_{1/z}$ with corresponding eigenfunction $z^2 g$, where $g(z)=\sum\limits_{n=0}^{\infty}a_n z^n\in H^2$. Then, 
    \begin{align*}
        A_{1/z}\cfrac{d}{dz}\Big[z^2g(z)\Big] &= \lambda (z^2g(z))\\
        \cfrac{1}{z} ~\cfrac{d}{dz}\Big[z^2 \sum\limits_{n=0}^{\infty}a_n z^n\Big] &= \lambda \Big(z^2 \sum\limits_{n=0}^{\infty}a_n z^n \Big)\\
        \sum\limits_{n=0}^{\infty}(n+2)a_n z^n &=  \lambda \Big( \sum\limits_{n=0}^{\infty}a_n z^{n+2} \Big)
    \end{align*}
    Then by comparing the corresponding coefficients we get that $a_n=0$ for all $n\in \mathbb{N}$. This implies that $g\equiv 0$, and the result follows.
\end{proof}

Note, the symbol $f(z)=1/z$, has no zeroes inside the disk and has a pole inside the disk. On the other hand, according to Proposition 5 in \cite{russo2022liouville}, if $f$ is an analytic function with no zeroes in a neighborhood of the closed disk, then every $\lambda \in \mathbb{C}$ is in the point spectrum with $g(z)=C~\text{exp}~\Big(\int_{0}^{z}\frac{\lambda}{f(w)}dw\Big)$ as the corresponding eigenfunction. We can further improve Proposition \ref{pro 5.1}, in the following way.

\begin{theorem}\label{Thm 5.2}
    Let $f$ be a non-analytic function with no zeroes in a neighborhood of the closed disc. Then the point spectrum of $\sigma(A_{f_{|B^2 H^2}})$ is empty.
\end{theorem}
\begin{proof}
    We will prove this by contradiction. Let $\lambda \in\sigma(A_{f_{|B^2 H^2}})$ be the point spectrum, then there exists a non-zero function $g\in H^2$ such that $A_{f_{|B^2 H^2}}(B^2g)(z)=\lambda (B^2g)(z)$. This means that
    \begin{equation}\label{Eq 2}
        g(z)=\cfrac{1}{B^2(z)}~\text{exp}~\Bigg(\int_{0}^{z}\cfrac{\lambda}{f(w)}dw\Bigg)
    \end{equation}
 where the above integral represent path integral from $0$ to $z$. Since $\lambda/f(z)$ has no poles inside the disk $\mathbb{D}$, hence it is bounded. On the other hand, by construction the poles of $f$ are zeroes of $B$, which means that $g$ is not analytic on the disk. Hence our assumption is wrong, and the result follows. 
\end{proof}

    It is well known that if $A:\text{Dom}(A)\rightarrow H$ is a closed operator, then
    
    \begin{theorem}\label{Thm 5.3}
      $\lambda \in \mathbb{C}$ belongs to the resolvent set $\rho (A)$ of $A$ if and only if $A-\lambda :\text{Dom}(A)\rightarrow H$ is bijective.  
    \end{theorem}

\begin{corollary}\label{Cor 5.4}
    The resolvent set of $A_{1/z}$ is empty. 
\end{corollary}
\begin{proof}
    Let $\lambda \in \rho(A_{1/z})$, where $A_{1/z}: z^2 H^2\rightarrow H^2$, then $A_{1/z}-\lambda I$ is an bijective operator.

    From Theorem \ref{Thm 5.2}, we know that; Ker ($A_{1/z}-\lambda I)=\{0\}$. Suppose $A_{1/z}-\lambda I$ is onto, then for any $h\in H^2$, there exists an $z^2 g\in z^2H^2$ such that

\begin{align*}
    (A_{1/z}-\lambda I)(z^2g) &= h\\
    \cfrac{1}{z}\cfrac{d}{dz}(z^2 g)-\lambda z^2 g &= h\\
    \cfrac{dg}{dz}+\bigg(\cfrac{2}{z}-\lambda z\bigg)g &= \cfrac{h}{z}
\end{align*}
    This is an First Order Linear ODE, whose integration factor is $z^2 e^{- \frac{\lambda z^2}{2}}$, and general solution is 
\begin{equation}\label{Eq 3}
    g(z)=\cfrac{1}{z^2 e^{- \frac{\lambda z^2}{2}}}\bigg[\int z e^{- \frac{\lambda z^2}{2}}~h(z)~dz+C\bigg] \text{~~for some~~ }C\in \mathbb{R}.
\end{equation}
 In particular, for $h(z)=1$ and $C=0$, we have $g(z)=-\frac{1}{\lambda z^2}$, which clearly doesn't belongs to $H^2$. Hence, $A_{1/z}-\lambda I$ is not onto for any $\lambda \in \mathbb{C}$. Therefore, $A_{1/z}-\lambda I$ is not a bijective operator, and the result follows.   
\end{proof}

A few remarks are in order.
\begin{enumerate}
    \item  \textbf{Spectrum of $A_{1/z}$ is continuous}. Combining Proposition \ref{pro 5.1} and Theorem \ref{Thm 5.2}, we can conclude that the spectrum of the operator $A_{1/z}$ is whole complex plane, i.e., $\sigma(A_{1/z})=\mathbb{C}$. This is in contrast with the case $A_z$ where $\sigma(A_z)=\mathbb{C}$, where the spectrum consists of only eigenvalues [Proposition 5, \cite{russo2022liouville}]. Hence, the spectrum of the operator $A_{f_{|B^2 H^2}}$ needs further investigations.\\
    
    \item \textbf{It is easy to see that $\text{Ran}(A_{1/z}-\lambda I)$ is non-empty}. From Theorem \ref{Thm 5.3}, the operator $A_{1/z}-\lambda I: z^2 H^2\rightarrow \text{~Ran~}(A_{1/z}-\lambda I)$ is bijective if and only if $h\in H^2$ satisfies equation \ref{Eq 2}. Now, consider the function $h(z)=e^{ \frac{\lambda z^2}{2}}, ~z\in \mathbb{D}$. Then $h\in H^{\infty}$, and from equation \ref{Eq 2} we have that
    $$
g(z)=\cfrac{1}{z^2 e^{- \frac{\lambda z^2}{2}}}\bigg[\int z e^{- \frac{\lambda z^2}{2}}~e^{ \frac{\lambda z^2}{2}}~dz\bigg]= \cfrac{1}{ e^{- \frac{\lambda z^2}{2}}} \in H^2
    $$
    This implies that $h(z)=e^{ \frac{\lambda z^2}{2}}\in \text{~Ran~}(A_{1/z}-\lambda I)$.\\

    \item \textbf{Does $z^n \in \text{Ran}(A_{1/z}-\lambda I)$ for any $n\geq 1$}? If it does, then from equation \ref{Eq 2}, we have that
    $$
g(z)=\cfrac{1}{z^2 e^{- \frac{\lambda z^2}{2}}}\bigg[\int z^{n+1} e^{- \frac{\lambda z^2}{2}}~dz\bigg] \in H^2.
    $$
    For \( n \geq 1 \), define:
\[
I_{n+1} = \int z^{n+1} e^{-\frac{\lambda z^2}{2}} \, dz.
\]
Using integration by parts with:
$u = z^{n}, \quad dv = z e^{-\frac{\lambda z^2}{2}} \, dz,$
we get:
$$
I_{n+1} = -\frac{z^{n}}{\lambda} e^{-\frac{\lambda z^2}{2}} + \frac{n}{\lambda} \int z^{n-1} e^{-\frac{\lambda z^2}{2}} \, dz.
$$
This yields the iterative formula:
$$
\boxed{I_{n+1} = -\frac{z^{n}}{\lambda} e^{-\frac{\lambda z^2}{2}} + \frac{n}{\lambda} I_{n-1}.}
$$
Also, for $z\in \mathbb{D}\text{~and~}\lambda \in \mathbb{C},$ we have that $ e^{-\frac{\lambda z^2}{2}}\in H^{\infty}$. Then 
$$
|I_0|= \Bigg|\int  e^{-\frac{\lambda z^2}{2}} \, dz \Bigg|\leq \pi \|I_0\|_{\infty}, \text{~~~and~~~} I_1=-\cfrac{1}{\lambda z^2}.
$$
Then, $g(z)=\cfrac{1}{z^2 e^{- \frac{\lambda z^2}{2}}}\Big[-\frac{z^{n}}{\lambda} e^{-\frac{\lambda z^2}{2}} + \frac{n}{\lambda} I_{n-1}\Big]= -\frac{z^{n-2}}{\lambda} + \frac{n}{\lambda z^2 e^{- \frac{\lambda z^2}{2}}} I_{n-1}.$ Clearly for $n=1$, we have $g(z)=-\frac{1}{\lambda z}+ \frac{1}{\lambda z^2 e^{- \frac{\lambda z^2}{2}}}I_0$. Hence, $z\notin \text{Ran}(A_{1/z}-\lambda I)$. Similarly, for $n=2$, we have $g(z)=-\frac{1}{\lambda }+ \frac{1}{\lambda z^2 e^{- \frac{\lambda z^2}{2}}}I_1$. Hence, $z^2 \notin \text{Ran}(A_{1/z}-\lambda I)$. Therefore, using the principal of mathematical induction on $n$, we can prove that $z^n \notin \text{Ran}(A_{1/z}-\lambda I)$. This observation forces us to ask the following natural question.



\item One might wonder about the case, where the non-analytic function $f$ can have zeroes inside the disk. From the proof of Theorem \ref{Thm 5.2}, in particular equation \ref{Eq 2}, we know that $\lambda \in \mathbb{C}$ is an eigenvalue of the operator $A_{f_{|B^2 H^2}}$ if and only if the corresponding eigenfunction is $B^2g$, where $g(z)=\frac{1}{B^2(z)}~\text{exp}~\Big(\int_{0}^{z}\frac{\lambda}{f(w)}dw\Big)\in H^2$, which is not possible as the poles of $f$ are zeroes of $B$. Therefore, similar to Theorem \ref{Thm 5.2}, \textbf{the point spectrum of  $A_{f_{|B^2 H^2}}$ is empty, even in the case of $f$ having zeroes inside the disk}. Additional to this observation, we will leave the investigation of continuous spectrum of the operator $A_{f_{|B^2 H^2}}$ for future work.

\end{enumerate}


\section{Learning the dynamics via Restricted Liouville Operator}\label{Sec 7}

In this section we will turn to our main goal of this paper, which is to present a method for parameter estimation of the equation \ref{eq 1}. From the Section \ref{Sec 6}, we learned that the Restricted Liouville Operator does not have any eigenfunctions. Therefore, the next best thing is to find the finite dimension subspace $\mathcal{M}\subset B^2H^2$, such that $A_{f_{|B^2 H^2}}\mathcal{M}\subseteq \mathcal{M}$. Let us first go through an example to see what can we expect in general. 

\begin{example}
    Suppose $f=1/z$, and let $\mathcal{M}\subset z^2H^2$ be the finite dimensional such that $A_{1/z}(\mathcal{M})\subseteq \mathcal{M}$, then we will prove that $\mathcal{M}=\{0\}$. 
\end{example}

\begin{proof}
    Suppose $\mathcal{M}\neq \{0\}$. Since $\mathcal{M}$ is finite dimensional, then there exist a least integer $m\geq 0$, and $g\in H^2$ such that 
    \[
z^2g\in \mathcal{M},\hspace{.2cm} \text{and}\hspace{.2cm}g(z)=\sum_{k=m}^{\infty}g_kz^k
\]
In other words, for all $h\in \mathcal{M}$ with $h(z)=\sum_{k=r}^{\infty}h_kz^k$, we have that $r\geq m+2$.

Now consider
\begin{align*}
    A_{1/z}(z^2g) &= \frac{1}{z}\frac{d}{dz}\left(z^2g(z)\right)\\
                  &= \frac{1}{z}\frac{d}{dz}\left(\sum_{k=m}^{\infty}g_kz^{k+2}\right)\\
                  &= \sum_{k=m}^{\infty}(k+2)g_kz^{k}\in \mathcal{M}.
\end{align*}
 This is a contradiction to the fact that the smallest power series in $\mathcal{M}$ starts from the integer $m+2$. Therefore our assumption is wrong, and hence $\mathcal{M}=\{0\}$.
\end{proof}

In general, we can prove the following result.

\begin{theorem}\label{Finite-dim}
    Let $\mathcal{M}\subset B^2H^2$ be a finite dimensional subspace such that $A_{f_{|B^2 H^2}}\mathcal{M}\subseteq \mathcal{M}$, then $\mathcal{M}=\{0\}$.
\end{theorem}
\begin{proof}
    We can use the similar line of arguments as above example to get the desired result.
\end{proof}

Theorem \ref{Finite-dim} makes it difficult to use the Projection method for the learning of the dynamics. Instead, we will achieve this by introducing a inner product on the symbols for which the operator $A_{f_{|B^2 H^2}}$ is densely defined via using the machinery we have build in the Section \ref{Sec 5}. 

Let $\gamma:[0,T]\rightarrow \mathbb{D}$ satisfies the dynamics of equation [1], i.e., $\dot{\gamma}= f(\gamma)$. Since by construction of Section [3], $B^2f$ is an analytic function on $\mathbb{D}$, therefore $B^2f$ is a more appropriate choice for the parametric identification than $f$. Consider $F= B^2f=\sum_{i=1}^{M}\theta_i Z_i$ for some $Z_i\in H^2$. 

It can be easily seen that the following identity holds
\[
A_{F_{|B^2H^2}}^{*}\Gamma_{\gamma}=\sum_{i=1}^{M}\theta_i A_{Z_{i_{|B^2H^2}}}^{*}\Gamma_{\gamma}.
\]

Now, our goal is to find the value of the parameter $\theta = (\theta_1,\theta_2,\ldots, \theta_M)\in \mathbb{R}^M$. Because of the above equation, we have
\begin{equation}\label{Eq 4}
    \Big \| A_{F_|B^2H^2}^{*}\Gamma_{\gamma}- \sum_{i=1}^{M}\theta_i A_{Z_{i_{|B^2H^2}}}^{*}\Gamma_{\gamma} \Big\|_2 ^2 = 0
\end{equation}
 For a collection, $\mathcal{F}$, of symbols of densely defined Liouville operators a bilinear form is given as $\langle B^2 h, B^2g\rangle_{\mathcal{F},\gamma}= \langle A_{h_{|B^2H^2}}^{*}\Gamma_{\gamma}, A_{g_{|B^2H^2}}^{*}\Gamma_{\gamma}\rangle$,  which gives a pre-inner product on the space of dynamical systems giving rise to densely defined Liouville operators over $B^2H^2$. From equation (4), we have
 \begin{equation}\label{Eq 5}
     \|A_{F_{|B^2H^2}}^{*}\Gamma_{\gamma}\|_2^2-2\sum_{i=1}^{M} \theta_i\langle B^2 F, B^2 Z_i\rangle_{\mathcal{F},\gamma}+ \sum_{i,j=1}^{M} \theta_i\theta_j\langle B^2 Z_i, B^2 Z_j\rangle_{\mathcal{F},\gamma} = 0.
 \end{equation}
 The above expression attains its minimum value if the gradient of equation \ref{Eq 5} is zero. Taking the derivative of equation \ref{Eq 5}, for any fixed value of $i=1,2,\ldots, M$, we get that
\[
-2 \langle B^2 F, B^2 Z_i\rangle_{\mathcal{F},\gamma} + 2\theta_i\langle B^2 Z_i, B^2 Z_i\rangle_{\mathcal{F},\gamma}+\sum_{j=1,j\neq i}^{M} \theta_j\langle B^2 Z_i, B^2 Z_j\rangle_{\mathcal{F},\gamma} = 0.
\]
Re-writing the above system of $M$ equations we get;

\[
\begin{pmatrix}
    \langle B^2 Z_1, B^2 Z_1\rangle_{\mathcal{F},\gamma}  &\cdots &\frac{1}{2}\langle B^2 Z_1, B^2 Z_M\rangle_{\mathcal{F},\gamma}\\
    
    \frac{1}{2}\langle B^2 Z_1, B^2 Z_1\rangle_{\mathcal{F},\gamma} &  \cdots & \frac{1}{2}\langle B^2 Z_2, B^2 Z_M\rangle_{\mathcal{F},\gamma}\\
\vdots & \ddots & \vdots\\
\frac{1}{2}\langle B^2 Z_M, B^2 Z_1\rangle_{\mathcal{F},\gamma} &  \cdots & \langle B^2 Z_M, B^2 Z_M\rangle_{\mathcal{F},\gamma}
\end{pmatrix}
\begin{pmatrix}
    \theta_1\\
    \theta_2\\
    \vdots\\
    \theta_M
\end{pmatrix}
= 
\begin{pmatrix}
    \langle B^2 F, B^2 Z_1\rangle_{\mathcal{F},\gamma}\\
    \langle B^2 F, B^2 Z_2\rangle_{\mathcal{F},\gamma}\\
    \vdots\\
    \langle B^2 F, B^2 Z_M\rangle_{\mathcal{F},\gamma}
\end{pmatrix}
\]
In compact form, we can represent the above system as;
\begin{equation}\label{psuedo}
    A_{[Z,Z]}~\theta = B_{[F,Z]}
\end{equation}
 Certainly, the value of $\theta$ depends upon our ability of efficiently computing the involved inner product in above system of equations. Using, Proposition \ref{pro 4.1} and \ref{pro 4.2}, it can be seen that
\begin{align*}
    \langle B^2 Z_i, B^2 Z_j\rangle_{\mathcal{F},\gamma} &= \langle A_{Z_{i_{|B^2H^2}}}^{*}\Gamma_{\gamma}, A_{Z_{j_{|B^2H^2}}}^{*}\Gamma_{\gamma}\rangle\\ 
    &= \int_{0}^{T}\int_{0}^{T} Z_{i}(\gamma(\tau)) ~\cfrac{\partial^2}{\partial w \partial z}K_{B^2}\big(\gamma(\tau), \gamma(t)\big)Z_{j}(\gamma(t))~d\tau dt
\end{align*}
By definition 
\begin{align*}
    A_{F_{|B^2H^2}}^{*}\Gamma_{\gamma} &= \int_{0}^{T} \cfrac{\partial}{\partial z}K_{B^2}(\cdot, \gamma(t))F(\gamma(t))dt\\
    &= \int_{0}^{T} \cfrac{\partial}{\partial w}\Big[B^2(\cdot) B^2(\gamma(t)) K(\cdot,\gamma(t))\Big]\overline{F(\gamma(t))}dt\\ 
    &= \int_{0}^{T} \cfrac{\partial}{\partial w}\Big[B^2(\cdot) B^2(\gamma(t)) K(\cdot,\gamma(t))\Big]\overline{B^2(\gamma(t))\dot{\gamma(t)}}dt
\end{align*}
Morever, the inner product on the Hardy space is invariant under the multiplication of the Blashcke factor on the boundary, then 
\[
A_{F_{|B^2H^2}}^{*}\Gamma_{\gamma}= B^2(\cdot)\big(K(\cdot, \gamma(T))-K(\cdot, \gamma(0))\big)
\]
where $K$ is the Szeg$\ddot{o}$ Kernel. Hence
\[
\langle B^2 F, B^2 Z_j\rangle_{\mathcal{F},\gamma} = \int_{0}^{T}A_{Z_{j_{|B^2H^2}}}\Big[B^2(\gamma(t))\Big(K(\gamma(t), \gamma(T))-K(\gamma(t), \gamma(0))\Big)\Big]~dt
\]

As the rows or columns $\langle B^2 F, B^2 Z_j\rangle_{\mathcal{F},\gamma}$ of the matrix $A_{[Z,Z]}$ can be linearly dependent, hence the solution of equation \ref{psuedo} is given by 
\[
\theta = A_{[Z,Z]}^{\dagger}~B_{[F,Z]}
\]
where $A_{[Z,Z]}^{\dagger}$ represents the Moore-Penrose inverse of $A_{[Z,Z]}$. 

\section{Acknowledgements}
This work was supported by the DOE/NIH/NSF Collaborative Research in Computational Neuroscience (CRCNS) program: \textit{Decomposing Neural Dynamics}, which provided research support to the second author. The authors also gratefully acknowledge the University of South Florida for providing institutional support, and office space, and an intellectually stimulating research environment that made this work possible.

\bibliographystyle{elsarticle-num} 
\bibliography{jmaa_references}     

\end{document}